\documentclass[12pt,a4paper]{article}
\usepackage{amsmath,amsthm,amssymb,mathtools,color}

\usepackage[utf8]{inputenc}
\usepackage[english]{babel}

\usepackage{enumerate}
\usepackage{url}
\usepackage[hidelinks]{hyperref}
\usepackage{upref}

\usepackage[a4paper,margin=2.5cm]{geometry}

\numberwithin{equation}{section}
\theoremstyle{plain}
\newtheorem{theorem}{Theorem}[section]

\newtheorem{lemma}[theorem]{Lemma}
\newtheorem{corollary}[theorem]{Corollary}

\theoremstyle{definition}

\newtheorem{remark}[theorem]{Remark}
\newtheorem{example}[theorem]{Example}

\newcommand{\abs}[1]{\lvert{#1}\rvert}  
\newcommand{\norm}[1]{\lVert{#1}\rVert} 
\newcommand{\ip}[2]{\langle{#1},{#2}\rangle} 

\DeclareMathOperator{\re}{Re} 

\DeclareMathOperator{\dist}{dist}
 
\newcommand{\any}{\mkern1mu\cdot\mkern1mu}
\newcommand{\bdz}{B_{\delta}(z)}
\newcommand{\bddz}{B_{\delta'}(z)}
\newcommand{\cO}{\mathcal{O}}
\newcommand{\cH}{\mathcal{H}}

\newcommand{\bC}{\mathbf{C}}

\newcommand{\bN}{\mathbf{N}}
\newcommand{\bZ}{\mathbf{Z}}

\newcommand{\eps}{\varepsilon}

\DeclareMathOperator{\Arg}{arg}
\newcommand{\e}{\textup{e}}

\newcommand{\be}[1]{\mathsf{e}_{#1}}

\title{On the directional growth of the resolvent norm}
\author{
Horia Cornean\footnote{Department of Mathematical Sciences, Aalborg University, Thomas Manns Vej 23, DK-9220 Aalborg~\O{}, Denmark. Email: cornean@math.aau.dk and matarne@math.aau.dk},\;
Henrik Garde\footnote{Department of Mathematics, Aarhus University, Ny Munkegade 118, DK-8000 Aarhus C, Denmark. Email: garde@math.au.dk},\; and
Arne Jensen\footnotemark[1]
}
\date{}

\begin{document}

\maketitle

\begin{abstract}
	
Let $A$ be a closed densely defined operator on a separable Hilbert space $\mathcal{H}$. Assume the resolvent set $\rho(A)$ is non-empty. For $z,z'\in\rho(A)$ let $[z,z']$ denote the straight line segment from $z$ to $z'$. For each $z\in\rho(A)$ we classify the behavior of the resolvent norm $\zeta\mapsto\norm{R_A(\zeta)}$ near $z$. Either there are $z'\in\rho(A)$, $z'\neq z$, $[z,z']\subset\rho(A)$, such that $\norm{R_A(\zeta)}\geq\norm{R_A(z)} + C\abs{\zeta-z}^\delta$ for $\zeta\in[z,z']$ with $\delta=1$ or $\delta=2$, or the function $\zeta\mapsto\norm{R_A(\zeta)}$ has a global minimum at $\zeta=z$. 

\end{abstract}

\section{Introduction and main results}

Let $\mathcal{H}$ be a separable Hilbert space and let $A$ be a closed densely defined operator on $\mathcal{H}$. Let $\rho(A)$ denote the resolvent set of $A$ and assume that $\rho(A)\neq\emptyset$. For $z\in \rho(A)$ the resolvent is denoted by $R_A(z)=(A-zI)^{-1}$. For $z,z'\in\rho(A)$ we denote by $[z,z']$ the straight line segment from $z$ to $z'$. 

For each $z\in\rho(A)$ we classify the behavior of the map $\zeta\mapsto\norm{R_A(\zeta)}$ near $z$. Either there are $z'\in\rho(A)$, $z'\neq z$, $[z,z']\subset\rho(A)$, such that 
\begin{equation*}
	\norm{R_A(\zeta)}\geq\norm{R_A(z)} + C\abs{\zeta-z}^\delta, \quad \zeta\in[z,z']
\end{equation*}
with $\delta=1$ or $\delta=2$, or the function $\zeta\mapsto\norm{R_A(\zeta)}$ has a global minimum at $\zeta=z$. In the latter case $\norm{R_A(\zeta)}$ may, or may not, equal this minimum value in an open neigborhood of $z$ in $\rho(A)$.

To state the results precisely we need the following lemma.

\begin{lemma}\label{lemma-abg}
	Let $z\in\rho(A)$. There exists a sequence $\{\psi_n\}_{n\in\bN}$ in $\cH$, such that $\norm{\psi_n}=1$, $n\in\bN$, and 
	\begin{equation}\label{R-lim}
		\lim_{n\to\infty}\norm{R_A(z)\psi_n}=\norm{R_A(z)}.
	\end{equation}
	Furthermore, the following limits exist:
	\begin{align}
		\alpha &=\lim_{n\to\infty} \ip{R_A(z)\psi_n}{R_A(z)^2\psi_n},\label{alpha}\\
		\beta &= \lim_{n\to\infty}\norm{R_A(z)^2\psi_n}^2,\label{beta}\\
		\gamma &= \lim_{n\to\infty}\ip{R_A(z)\psi_n}{R_A(z)^3\psi_n}.\label{gamma} 
	\end{align}
\end{lemma} 

\begin{theorem}\label{revthm}
	Let $z\in\rho(A)$. Let $\alpha$, $\beta$, and $\gamma$ be obtained from Lemma~\ref{lemma-abg}. We then have the following results.
	\begin{itemize}
		\item[\textup{(i)}] Assume $\alpha\neq0$. For any $\theta\in (-\tfrac{\pi}{2},\tfrac{\pi}{2})$, there exist $C>0$ and $\eps_0>0$, such that for $z'$ with $\abs{z'-z}\leq\eps_0$ and $\Arg(z'-z) = \theta - \Arg\alpha$ we have
		\begin{equation}\label{result-i}
			\norm{R_A(\zeta)} \geq \norm{R_A(z)} + C\abs{\zeta-z},\quad \zeta\in[z,z'].
		\end{equation}
		\item[\textup{(ii)}] Assume $\alpha=0$ and $\gamma\neq0$. For any $\theta\in (-\tfrac{\pi}{4},\tfrac{\pi}{4})\cup(\tfrac{3\pi}{4},\tfrac{5\pi}{4})$, there exist $C>0$ and $\eps_0>0$, such that for $z'$ with $\abs{z'-z}\leq\eps_0$ and $\Arg(z'-z) = \theta - \tfrac{1}{2}\Arg\gamma$ we have
		\begin{equation}\label{result-ii}
			\norm{R_A(\zeta)} \geq \norm{R_A(z)} + C\abs{\zeta-z}^2,\quad \zeta\in[z,z'].
		\end{equation}
		\item[\textup{(iii)}] Assume $\alpha=0$, $\gamma=0$, and $\beta>0$. Then there exist $C>0$ and $\eps_0>0$ such that for $\zeta\in\rho(A)$ we have
		\begin{equation}\label{result-iii}
			\norm{R_A(\zeta)} \geq \norm{R_A(z)} + C\abs{\zeta-z}^2,\quad \abs{\zeta-z}\leq\eps_0.
		\end{equation}
		\item[\textup{(iv)}] Assume $\beta=0$. Then we have
		\begin{equation}\label{result-iv}
			\norm{R_A(\zeta)}\geq\norm{R_A(z)}\quad \text{for all $\zeta\in\rho(A)$}.
		\end{equation}
	\end{itemize} 
\end{theorem}

Part (iv) was proved in \cite{BS}, and is included in the statement of the theorem for completeness. Note that $\beta=0$ implies $\alpha=0$ and $\gamma=0$. Thus for any $z\in\rho(A)$ one of the four possibilities for $\alpha$, $\beta$, and $\gamma$ stated in the theorem holds. The new results here are the growth estimates and the condition for an isolated minimum in part (iii). Furthermore, no assumption besides $\rho(A)\neq\emptyset$ is needed on $A$.

\begin{corollary}\label{cor}
	Assume that $A$ is a closed densely defined operator on a Hilbert space $\cH$ with non-empty resolvent set. Let $z\in\rho(A)$. Assume there exists $\psi\in\cH$, $\norm{\psi}=1$, such that 
	\begin{equation*}
		\norm{R_A(z)\psi}=\norm{R_A(z)}.
	\end{equation*}
	Then $\norm{R_A(\any)}$ is non-constant in a neighborhood of $z$, and one of the growth estimates in \textup{(i)--(iii)} holds in this neighborhood. 
\end{corollary}

The question whether the resolvent norm can be constant in an open subset of $\rho(A)$ has been extensively studied in Banach spaces. We refer to~\cite{S2,DS} and references therein. In particular, in~\cite[Theorem~3.2]{S2} an example is given of a closed densely defined operator on the Hilbert space $\ell^2(\bN)$ with constant resolvent norm in a neighborhood of $z=0$. In~\cite{BS} it is shown that in case of constant resolvent norm in an open set, this constant value is a global minimum of the resolvent norm.

Several classes of operators with non-constant resolvent norm in any open subset of $\rho(A)$ have been found in complex uniformly convex Banach spaces. They include all bounded operators, generators of $C_0$-semigroups, and operators with compact resolvent, see e.g.~\cite{BGS,DS,S2,S}. For examples of operators with non-constant resolvent norm in any open subset of $\rho(A)$, but not belonging to one of the three classes mentioned, see~\cite{BS}.

In the Hilbert space case considered here, Corollary~\ref{cor} implies that if $A$ has compact resolvent then the resolvent norm is non-constant on all open subsets of $\rho(A)$. This follows since $\norm{R_A(z)}^2$ is an eigenvalue of $R_A(z)^{\ast}R_A(z)$ for all $z\in\rho(A)$. As a new result we additionally get a growth estimate for the resolvent norm.

Our results give explicit criteria for non-constant resolvent norm. The quantities $\alpha$, $\beta$, and $\gamma$ allow one to construct explicit operators satisfying one of the four conditions in the theorem, but not belonging to three general classes mentioned above. Our examples are weighted shift operators on $\ell^2(\bZ)$. All operators in the examples are unbounded and non-normal with non-compact resolvent.

Let us note that in the unpublished article \cite{CGJK} we obtained some weaker results on directional growth of the resolvent norm than those presented here.

In Section~\ref{sect2} we give the proofs of our results. Section~\ref{sect3} is devoted to examples and remarks.

\section{Proofs}\label{sect2}

\begin{proof}[Proof of Lemma~\ref{lemma-abg}]
	Let $z\in\rho(A)$. There exists a sequence $\{\phi_k\}_{k\in\bN}$ in $\cH$, such that $\norm{\phi_k}=1$, $k\in\bN$, and $\lim_{k\to\infty}\norm{R_A(z)\phi_k}=\norm{R_A(z)}$. 
	The three sequences 
	\begin{align}
		&\{\ip{R_A(z)\phi_k}{R_A(z)^2\phi_k}\}_{k\in\bN}, \label{eq:seq1} \\
		&\{\norm{R_A(z)^2\phi_k}^2\}_{k\in\bN}, \label{eq:seq2} \\
		&\{\ip{R_A(z)\phi_k}{R_A(z)^3\phi_k}\}_{k\in\bN}, \label{eq:seq3}
	\end{align}
	are bounded in $\bC$. Hence, there is a subsequence $\{\phi_{k_j}\}_{j\in\bN}$ for which \eqref{eq:seq1} converges, a subsequence $\{\phi_{{k_j}_m}\}_{m\in\bN}$ for which \eqref{eq:seq2} converges, and finally a subsequence $\{\phi_{{{k_j}_m}_n}\}_{n\in\bN}$ for which \eqref{eq:seq3} converges. The result therefore follows with $\psi_n = \phi_{{{k_j}_m}_n}$, leading to the existence of all three limits~\eqref{alpha}--\eqref{gamma}. 
\end{proof}

\begin{proof}[Proof of Theorem~\ref{revthm}]
	Let $z\in\rho(A)$ be fixed.	We use the notation 
	\begin{equation*}
		\bdz = \{ w\in\bC \mid \abs{w-z}<\delta \}.
	\end{equation*}
	Let $z\in\rho(A)$ and choose $\delta>0$ such that $\bdz\subset\rho(A)$. Iterating the first resolvent identity we obtain the expansion
	\begin{align}\label{taylor}
		R_A(\zeta) &= R_A(z)+(\zeta-z)R_A(z)^2	+ (\zeta-z)^2R_A(z)^3 \notag\\
		&\phantom{={}}+(\zeta-z)^3\bigl(I-(\zeta-z)R_A(z)\bigr)^{-1}R_A(z)^4.
	\end{align}
	The last term is of order $\cO(\abs{\zeta-z}^3)$, with a constant independent of $\zeta\in\bddz$, where $\delta'\in(0,\delta)$. We compute $R_A(\zeta)^{\ast}R_A(\zeta)$ using \eqref{taylor}:
	\begin{align*}
		R_A(\zeta)^{\ast}R_A(\zeta)&=R_A(z)^{\ast}R_A(z) + (\zeta-z)R_A(z)^{\ast}R_A(z)^2 + \overline{(\zeta-z)}(R_A(z)^{\ast})^2R_A(z) \\
		&\phantom{={}} + \abs{\zeta-z}^2(R_A(z)^{\ast})^2R_A(z)^2 + (\zeta-z)^2R_A(z)^{\ast}R_A(z)^3 \\
		&\phantom{={}} + \overline{(\zeta-z)}^2(R_A(z)^{\ast})^3R_A(z) + \cO(\abs{\zeta-z}^3).
	\end{align*}
	Let $\{\psi_n\}_{n\in\bN}$ denote a sequence obtained from Lemma~\ref{lemma-abg}, such that \eqref{alpha}--\eqref{gamma} hold. Then we have
	\begin{align}
		\norm{R_A(\zeta)}^2 &\geq \norm{R_A(\zeta)\psi_n}^2	= \ip{\psi_n}{R_A(\zeta)^{\ast}R_A(\zeta)\psi_n} \notag\\
		&= \norm{R_A(z)\psi_n}^2 \notag\\
		&\phantom{={}} + 2\re\bigl[(\zeta-z)\ip{R_A(z)\psi_n}{R_A(z)^2\psi_n}\bigr] \notag\\
		&\phantom{={}} + \abs{\zeta-z}^2\norm{R_A(z)^2\psi_n}^2 \notag\\
		&\phantom{={}} + 2\re\bigl[(\zeta-z)^2\ip{R_A(z)\psi_n}{R_A(z)^3\psi_n}\bigr] + \cO(\abs{\zeta-z}^3).\label{R-estimate}
	\end{align}
	Note that we have an explicit expression for the error term $\cO(\abs{\zeta-z}^3)$, obtained from the last term in \eqref{taylor} and some straightforward computations. Since the error term is bounded by $C\abs{\zeta-z}^3$ with $C$ independent of $n$, and the remaining terms on the right hand side of \eqref{R-estimate} converge as $n\to\infty$, then for $\zeta\in\bddz$ we have
	\begin{equation}\label{lim-estimate}
		\norm{R_A(\zeta)}^2 \geq \norm{R_A(z)}^2 + 2\re\bigl[(\zeta-z)\alpha\bigr] + \abs{\zeta-z}^2\beta + 2\re\bigl[(\zeta-z)^2\gamma\bigr] + \cO(\abs{\zeta-z}^3).
	\end{equation}
	Let $x>0$ be fixed and let $\xi$ be a small parameter. Then Taylor's theorem yields
	\begin{equation*}
		\sqrt{x^2+\xi} = x+\frac{1}{2x}\xi - \frac{1}{8x^3}\xi^2 + \cO(\xi^3).
	\end{equation*}
	Applying this expansion to \eqref{lim-estimate} we obtain
	\begin{align}
		\norm{R_A(\zeta)} &\geq \norm{R_A(z)} \notag \\
		&\phantom{={}} + \frac{1}{2\norm{R_A(z)}} \bigl[2\re\bigl((\zeta-z)\alpha\bigr)+\abs{\zeta-z}^2\beta + 2\re\bigl((\zeta-z)^2\gamma\bigr)\bigr] \notag\\
		&\phantom{={}} - \frac{1}{8\norm{R_A(z)}^3} \bigl[2\re\bigl((\zeta-z)\alpha\bigr)\bigr]^2 + \cO(\abs{\zeta-z}^3). \label{new-eq2}
	\end{align}
	
	We can now prove Theorem~\ref{revthm}(i). Assume $\alpha\neq0$. Let $\theta\in(-\tfrac{\pi}{2},\tfrac{\pi}{2})$ which implies $\cos(\theta)>0$. Let $\eps_0\in(0,\delta')$ and $\nu = \theta-\Arg \alpha$. Define $z' = z + \eps_0\e^{i\nu}$, then $[z,z']\subset\bddz\subset\rho(A)$. For $\zeta\in[z,z']$ we have $\re((\zeta-z)\alpha) = \abs{\zeta-z}\abs{\alpha}\cos(\theta)$. Choosing $\eps_0$ sufficiently small we obtain that \eqref{result-i} holds.
	
	Next we prove Theorem~\ref{revthm}(ii). Assume $\alpha=0$ and $\gamma\neq0$. Let $\theta\in(-\tfrac{\pi}{4},\tfrac{\pi}{4})\cup(\tfrac{3\pi}{4},\tfrac{5\pi}{4})$ which implies $\cos(2\theta)>0$. Let $\eps_0\in(0,\delta')$ and $\nu = \theta-\tfrac{1}{2}\Arg \gamma$. Define $z' = z + \eps_0\e^{i\nu}$, then $[z,z']\subset\bddz\subset\rho(A)$. For $\zeta\in[z,z']$ we have $\re((\zeta-z)^2\gamma) = \abs{\zeta-z}^2\abs{\gamma}\cos(2\theta)$. Since $\beta\geq0$, the estimate \eqref{result-ii} follows for a sufficiently small $\eps_0$.
	
	Assume $\alpha=0$, $\gamma=0$, and $\beta>0$. Then Theorem~\ref{revthm}(iii) follows immediately from~\eqref{new-eq2}.
	
	The result in part (iv) was proved in \cite{BS}. We give a variant of their argument here. Assume $\beta=0$ and let $\{\psi_n\}_{n\in\bN}$ be a sequence obtained from Lemma~\ref{lemma-abg}. To prove Theorem~\ref{revthm}(iv) we go back to \eqref{taylor}. Then for any $\zeta\in\rho(A)$ we have
	\begin{align*}
		\norm{R_A(\zeta)}&\geq \norm{R_A(\zeta)\psi_n} \geq\norm{R_A(z)\psi_n} -\Bigl[\abs{\zeta-z}+\abs{\zeta-z}^2\norm{R_A(z)} \\
		&\quad + \abs{\zeta-z}^3\bigl\lVert \bigl(I-(\zeta-z)R_A(z)\bigr)^{-1}R_A(z)^2\bigr \rVert \Bigr]\norm{R_A(z)^2\psi_n}.
	\end{align*}
	The estimate \eqref{result-iv} then follows by taking the limit $n\to\infty$.
\end{proof}

\begin{proof}[Proof of Corollary~\ref{cor}]
	Let $z\in\rho(A)$. Assume $\psi\in\cH$, $\norm{\psi}=1$, such that $\norm{R_A(z)\psi}=\norm{R_A(z)}$. Let $\psi_n=\psi$, $n\in\bN$, be the constant sequence such that Lemma~\ref{lemma-abg} holds. We have $\beta=\norm{R_A(z)^2\psi}$. Since $R_A(z)$ is injective, we have $\beta>0$ and the result follows.
\end{proof}

\section{Examples and remarks}\label{sect3}

Let $\cH=\ell^2(\bZ)$.  We construct examples of operators $A$ on $\cH$ satisfying the four cases in Theorem~\ref{revthm}.

We assume $0$ is in the resolvent set of $A$ and construct examples where $A$ is a weighted shift operator. We start by constructing the bounded operator $B=R_A(0)$. Let $b_k\in\bC\setminus\{0\}$, $k\in\bZ$, be a bounded sequence. We use the notation $\{\be{k}\}_{k\in\bZ}$ for the canonical basis in $\cH$, i.e.~$(\be{k})_j=\delta_{j,k}$, $j,k\in\bZ$. Define $B$ by
\begin{equation}\label{eq41}
	B\be{k}=b_{k+1}\be{k+1},\quad k\in\bZ.
\end{equation}
Then $B$ is invertible with
\begin{equation}\label{eq42}
	B^{-1}\be{k}=b_{k}^{-1}\be{k-1},\quad k\in\bZ,
\end{equation}
such that in our examples $A=B^{-1}$ is defined by \eqref{eq42}. On the maximal domain,
\begin{equation*}
	\bigl\{ x\in\cH \bigm| \sum_{k\in\bZ}\abs{b_{k}}^{-2}\abs{x_k}^2<\infty \bigr\},
\end{equation*}
the operator $A$ is a closed densely defined operator. We have
\begin{equation*}
	B^{\ast}\be{k}=\overline{b}_{k}\be{k-1},\quad k\in\bZ.
\end{equation*}
Simple computations yield the results
\begin{equation}
	(B^{\ast}B)\be{k}=\abs{b_{k+1}}^2\be{k} \quad\text{and}\quad (BB^{\ast})\be{k}=\abs{b_{k}}^2\be{k},\label{eq44}
\end{equation}
for $k\in\bZ$. Thus $B$ is normal if and only if $\{\abs{b_k}\}_{k\in\bZ}$ is a constant sequence. Furthermore, we have 
\begin{equation*}
	\norm{B}=\sup_{k\in\bZ}\abs{b_k}.
\end{equation*}
In order to give the examples, we need the following results. The straightforward verifications are omitted. For $k\in\bZ$ we have
\begin{align}
	B^{\ast}B^2\be{k}&=b_{k+1}\abs{b_{k+2}}^2\be{k+1},\label{eq47} \\
	B^{\ast}B^3\be{k}&=b_{k+1}b_{k+2}\abs{b_{k+3}}^2\be{k+2},\label{eq48} \\
	(B^{\ast})^2B^2\be{k}&=\abs{b_{k+1}}^2\abs{b_{k+2}}^2\be{k}.\label{eq49}
\end{align}

In the examples below, we make different choices for the bounded sequence $\{b_k\}_{k\in\bZ}$ of non-zero complex numbers and a sequence $\{\psi_k\}_{k\in\bZ}$ of unit vectors in $\cH$, to satisfy the conditions in each of the four cases in Theorem~\ref{revthm}. In each example we will have $\liminf_{\abs{k}\to\infty}\abs{b_k}=0$ and $\limsup_{\abs{k}\to\infty}\abs{b_k}>0$ such that $A$ is non-normal, unbounded, and has non-compact resolvent.

\begin{example}\label{ex43}
	Assume $c\neq0$ and let the non-zero sequence $\{b_k\}_{k\in\bZ}$ satisfy
	\begin{equation*}
		\sup_{k\in\bZ}\abs{b_k}\leq\abs{c}, \quad \lim_{\abs{k}\to\infty}b_{3k}=\lim_{\abs{k}\to\infty}b_{3k+1}=c, \quad\text{and}\quad \lim_{\abs{k}\to\infty}b_{3k+2}=0.
	\end{equation*}
	Define $\psi_k=2^{-1/2}(\be{3k-1}+\be{3k})$ such that $\norm{\psi_k}=1$. We have $\norm{B}\leq\abs{c}$ and
	\begin{equation*}
		\norm{B\psi_k}^2=\ip{\psi_k}{B^{\ast}B\psi_k}=\tfrac{1}{2}(\abs{b_{3k}}^2 + \abs{b_{3k+1}}^2).
	\end{equation*}
	The assumptions on $\{b_k\}_{k\in\bZ}$ imply that \eqref{R-lim} holds. Using \eqref{eq47} we find
	\begin{equation*}
		\ip{\psi_k}{B^{\ast}B^2\psi_k}=\tfrac{1}{2}	b_{3k}\abs{b_{3k+1}}^2,
	\end{equation*}
	such that by \eqref{alpha}
	\begin{equation*}
		\alpha=\lim_{\abs{k}\to\infty} \tfrac{1}{2}b_{3k}\abs{b_{3k+1}}^2=\tfrac{1}{2}c\abs{c}^2\neq0.
	\end{equation*}
	By Theorem~\ref{revthm}(i), for any $\theta\in(-\tfrac{\pi}{2},\tfrac{\pi}{2})$ there exists a $z'$ such that \eqref{result-i} holds with $z=0$ and $\Arg(z')=\theta-\Arg c$.
\end{example}

\begin{example}\label{ex44}
	Assume $c\neq0$ and let the non-zero sequence $\{b_k\}_{k\in\bZ}$ satisfy
	\begin{equation*}
		\sup_{k\in\bZ}\abs{b_k}\leq\abs{c}, \quad \lim_{\abs{k}\to\infty}b_{4k+1}=\lim_{\abs{k}\to\infty}b_{4k+2} = \lim_{\abs{k}\to\infty}b_{4k+3} = c, \quad\text{and}\quad \lim_{\abs{k}\to\infty}b_{4k}=0.
	\end{equation*}
	Define $\psi_k=2^{-1/2}(\be{4k}+\be{4k+2})$ such that $\norm{\psi_k}=1$. We have $\norm{B}\leq\abs{c}$ and
	\begin{equation*}
		\norm{B\psi_k}^2=\ip{\psi_k}{B^{\ast}B\psi_k}=\tfrac{1}{2}(\abs{b_{4k+1}}^2 + \abs{b_{4k+3}}^2).
	\end{equation*}
	The assumptions on $\{b_k\}_{k\in\bZ}$ imply that \eqref{R-lim} holds. Using \eqref{eq47} we find
	\begin{equation*}
		\ip{\psi_k}{B^{\ast}B^2\psi_k}=0,
	\end{equation*}
	such that by \eqref{alpha} we have $\alpha=0$. Using \eqref{eq48} we get
	\begin{equation*}
		\ip{\psi_k}{B^{\ast}{B^3}\psi_k}=\tfrac{1}{2} b_{4k+1}b_{4k+2}\abs{b_{4k+3}}^2
	\end{equation*}
	such that 
	\begin{equation*}
		\gamma=\lim_{\abs{k}\to\infty}	\tfrac{1}{2} b_{4k+1}b_{4k+2}\abs{b_{4k+3}}^2 = \tfrac{1}{2}c^2\abs{c}^2\neq 0.
	\end{equation*}
	By Theorem~\ref{revthm}(ii), for any $\theta\in(-\tfrac{\pi}{4},\tfrac{\pi}{4})\cup(\tfrac{3\pi}{4},\tfrac{5\pi}{4})$ there exists a $z'$ such that \eqref{result-ii} holds with $z=0$ and $\Arg(z')=\theta-\Arg c$.
\end{example}

\begin{example}\label{ex41}
	Assume that the non-zero sequence $\{b_k\}_{k\in\bZ}$ satisfies
	\begin{equation*}
		\abs{b_0}\geq\sup_{k\neq0}\abs{b_k}, \quad \limsup_{\abs{k}\to\infty}\abs{b_k}>0, \quad\text{and}\quad \liminf_{\abs{k}\to\infty}\abs{b_k}=0.
	\end{equation*}
	It follows from \eqref{eq44} that $\norm{B}=\abs{b_0}$, and from \eqref{eq41} that $\norm{B\be{-1}}=\norm{B}$. Thus we can use Lemma~\ref{lemma-abg} and Theorem~\ref{revthm} with the constant sequence
	$\psi_k=\be{-1}$. Then \eqref{eq47} and \eqref{eq48} show that $\alpha=0$ and $\gamma=0$. From \eqref{beta} and \eqref{eq49} follows 
	\begin{equation*}
		\beta=\norm{B^2\be{-1}}^2=\abs{b_0b_1}^2>0.
	\end{equation*}
	Thus Theorem~\ref{revthm}(iii) shows that $\zeta\mapsto\norm{R_A(\zeta)}$ has an isolated local minimum at $\zeta=0$. 
\end{example}

\begin{example}
	Assume that the non-zero sequence $\{b_k\}_{k\in\bZ}$ satisfies
	\begin{equation*}
		0<b=\sup_{k\in\bZ}\abs{b_k}, \quad b>\abs{b_k} \text{ for all } k\in\bZ, \quad \lim_{\abs{k}\to\infty}\abs{b_{2k}} = b, \quad\text{and}\quad \lim_{\abs{k}\to\infty}b_{2k+1} = 0.
	\end{equation*}
	Let $\psi_k=\be{2k-1}$, $k\in\bZ$. Then \eqref{R-lim} holds for this sequence. A short computation shows that $\alpha=\gamma=0$. We have
	\begin{equation*}
		\beta=\lim_{\abs{k}\to\infty}\norm{B^2\be{2k-1}}^2 = \lim_{\abs{k}\to\infty}\abs{b_{2k}}^2\abs{b_{2k+1}}^2=0.
	\end{equation*}
	Thus Theorem~\ref{revthm}(iv) implies that $\zeta\mapsto\norm{R_A(\zeta)}$ has a global minimum at $\zeta=0$. We know from general results that $A$ cannot be a bounded operator. In this example it follows explicitly from $\liminf_{\abs{k}\to\infty}\abs{b_k}=0$ that $A$ is unbounded.
\end{example}

\begin{remark}\label{rem30}
	Let $A$ be a bounded operator on $\cH$ and let $z\in\rho(A)$. Let $\beta$ be obtained from Lemma~\ref{lemma-abg} with the sequence $\{\psi_n\}_{n\in\bN}$. Assume $\beta=0$. Then
	\begin{equation*}
		\lim_{n\to\infty}\norm{R_A(z)\psi_n} = \lim_{n\to\infty}\norm{(A-zI)R_A(z)^2\psi_n} = 0,
	\end{equation*}
	since $A$ is bounded. But this result contradicts \eqref{R-lim}. Thus we get a different proof of the result that the resolvent norm of a bounded operator is non-constant in any open subset of $\rho(A)$. See e.g.~\cite{BS} for the standard argument.
\end{remark}

\begin{remark}\label{rem31}
	Let $A$ be a closed densely defined operator on $\cH$. Let $z\in\rho(A)$ and assume that $\norm{R_A(z)}^2$ is an eigenvalue of the operator $R_A(z)^{\ast}R_A(z)$. Let $\psi$ be a normalized eigenvector. Then Corollary~\ref{cor} holds. In particular, if $R_A(z_0)$ is compact for some $z_0\in\rho(A)$, then Corollary~\ref{cor} holds at any $z\in\rho(A)$. Thus in the Hilbert space case we give a different proof of the result from \cite{DS} that an operator with compact resolvent cannot have constant resolvent norm in any open subset of $\rho(A)$, and add the result that we have a minimal growth rate in some directions.
\end{remark}

\begin{remark}\label{rem32}
	Assume that $A$ is a normal operator on $\cH$. Then \eqref{result-i}  holds for all $z\in\rho(A)$. This can be seen directly from the following explicit estimate. We have
	\begin{equation*}
		\norm{R_A(z)}=\frac{1}{\dist(z,\sigma(A))}.
	\end{equation*}
	Let $z_0\in\sigma(A)$ such that 
	$\dist(z,\sigma(A))=\abs{z-z_0}$. Let $z'=(z+z_0)/2$. Then for $\zeta\in[z,z']$ we have $|\zeta-z_0|=|z-z_0|-|\zeta-z|$ and
	\begin{align*}
		\norm{R_A(\zeta)} &\geq \frac{1}{\abs{\zeta-z_0}} = \frac{1}{\abs{z-z_0}-\abs{\zeta-z}} = \frac{\abs{z-z_0} + \abs{\zeta-z}}{\abs{z-z_0}^2-\abs{\zeta-z}^2} \notag\\
		&\geq \frac{\abs{z-z_0} + \abs{\zeta-z}}{\abs{z-z_0}^2} = \norm{R_A(z)}+\frac{1}{\dist(z,\sigma(A))^2}\abs{\zeta-z}. 
	\end{align*}
\end{remark}

\subsubsection*{Acknowledgements} 

We thank two referees for comments and suggestions that lead to the present version of our results. We also thank one of them for pointing out \cite{BS} to us.

HG is supported by grant 10.46540/3120-00003B from Independent Research Fund Denmark.


\begin{thebibliography}{0}

\setlength{\itemsep}{-2mm} \footnotesize

\bibitem{BS} Bögli, S., and Siegl, P., Remarks on the convergence of pseudospectra, \emph{Integr. Equ. Oper. Theory} 80 (2014), 303--321.

\bibitem{BGS} Böttcher, A., Grudsky, S. M., and Silbermann, B., Norms of inverses, spectra, and pseudospectra of large truncated Wiener-Hopf operators and Toeplitz matrices. 
\emph{New York J. Math.} 3 (1997), 1--31.

\bibitem{CGJK}
Cornean, H., Garde, H., Jensen, A., and Knörr, H.~K., A local directional growth estimate of the resolvent norm, arXiv:1808.04419~[math.SP]. Unpublished.

\bibitem{DS} Davies, E.~B., and Shargorodsky, E., 
{Level sets of the resolvent norm of a linear operator revisited}. 
\emph{Mathematika} 62 (2016), no. 1, 243--265.

\bibitem{S2} 
Shargorodsky, E.,
{On the level sets of the resolvent norm of a linear operator}.
\emph{Bull. London Math. Soc.} 40 (2008), 493--504.

\bibitem{S} 
Shargorodsky, E.,
{On the definition of pseudospectra}. 
\emph{Bull. London Math. Soc.} 41 (2009), no. 3, 524--534.

\end{thebibliography}
\end{document}